\documentclass[11pt]{article}
\usepackage[dvips]{epsfig}
\usepackage{amsmath}
\usepackage{amsfonts}
\usepackage{amssymb}
\usepackage{graphicx}
\usepackage{psfrag}
\usepackage{bbm}
\usepackage{float}
\usepackage{latexsym}
\newtheorem{thm}{Theorem}[section]
\newtheorem{coro}{Corollary}[section]
\newtheorem{lm}{Lemma}[section]
\newtheorem{prop}{Proposition}[section]

\newtheorem{rem}{Remark}[section]
\numberwithin{equation}{section}
\newenvironment{proof}[1][Proof]{\noindent\textbf{#1.} }{\ \rule{0.5em}{0.5em}}

\renewcommand{\phi}{\varphi}

\renewcommand{\chi}{\mathcal{X}}

\begin{document}
\title{Evolving hypersurfaces by their mean curvature in the background manifold evolving by Ricci flow\footnote{This research was supported by Natural Science Foundation of China, Grant No. 11131007, and Zhejiang Provincial Natural Science Foundation of China, Grant No. LY14A010019. }}
\author{{ Weimin Sheng \ \ and \ \ Haobin Yu} }
\date{}
\maketitle

\begin{abstract}
We consider the problem of deforming a one-parameter family of hypersurfaces
immersed into  closed Riemannian manifolds with positive curvature operator.
The  hypersurface in this family satisfies mean curvature flow while the ambient metric
satisfying the normalized Ricci flow.
We prove that if the initial metric of the background
manifold is  sufficiently pinched and
the initial hypersurface  also
satisfies a suitable pinching condition, then either the hypersurfaces shrink to a round point in finite time or
converge to a totally geodesic sphere as the time tends to infinity.

\end{abstract}

\vspace{1mm}{\footnotesize \textbf{Keywords}:\hspace{2mm} mean curvature flow,
normalized Ricci flow, totally geodesic  sphere}

\section{Introduction}

\quad Let $(N^{n+1},\bar{g})$ be a complete, simply connected Riemannian
manifold, $X(\cdot ,t)$: $M^{n}\rightarrow N^{n+1}$ be a one-parameter family
of smooth oriented hypersurface immersions, satisfying the evolution equation
\begin{equation}
\left\{
\begin{aligned}
\frac{\partial X(x,t)}{\partial t}&=-H(x,t)\nu(x,t),x\in M^{n},t>0\\
X(\cdot ,t)&=X_{0},\\
\end{aligned}
\right.
\end{equation}
where $H(x,t)$ is the mean curvature of the hypersurface $
X(\cdot ,t)$ at the point $X(x,t)$, $\nu (x,t)$ is the outer unit normal to $
X(\cdot ,t)$ and $X_{0}$ is a given oriented hypersurface in $N^{n+1}.$ This
is the well-known mean curvature flow which has been studied extensively,
when the  background is a fixed Riemannian manifold,
 see \cite{Co,Hu1,Hu3,Hu5,SW,Wa} for instance.

In \cite{Hu5}, Huisken got an important monotonicity formula for hypersurfaces
in the Gaussian shrinker background. So it is reasonable to consider the mean curvature flow in a moving ambient space.
In particular, when the metric of $N^{n+1}$ satisfies the
Ricci flow,  we call the coupled evolutions as the "Ricci-Mean curvature flow".
Magni-Mantegazza-Tsatis\cite{Ma} showed a similar  monotonicity as Huisken's for
mean curvature flow in a gradient Ricci soliton background.
Recently, John lott \cite{Lo} presented a very valuable explanation on the "Ricci-Mean curvature flow".
He used the variation method to get the evolution equations of the second
fundamental form and the mean curvature. In the case of $N^{n+1}$ being a gradient Ricci soliton,
he introduced the concept of mean curvature soliton which can be regarded as the generalization of self-shrinker. In \cite{HL}, Han and Li studied a surface immersed in a K\"ahler surface evolved by its mean curvature flow while the K\"ahler surface evolved by K\"ahler-Ricci flow. They proved if the K\"ahler surface is sufficiently close to a K\"ahler-Einstein surface and the initial surface is sufficiently close to a holomorphic curve, then the surface converges to a holomorphic curve along the K\"ahler-Ricci mean curvature flow. This is the first convergence result on Ricci-Mean curvature flow.

In this paper, we consider a one-parameter family of immersions
$X(\cdot,t): M^n\rightarrow (N^{n+1},\bar{g}(t))$, which satisfies
\begin{equation}\label{1.2}
\left\{
\begin{aligned}
&\frac{\partial X(x,t)}{\partial t}=-H(x,t)\nu(x,t),\ \ \ \ x\in M^{n},t>0 \\
&\frac{\partial \bar{g}(t)}{\partial t}=-2\overline{Ric}(t)+\frac{2\bar{r}}{n+1}\bar{g}(t),\ \ \  \bar{g}(0)=\bar{g}_0\\
\end{aligned}
\right.
\end{equation}
where $\bar{r}$ is the average of the scalar curvature of the background  metric $\bar{g}$.
In \cite{Hu4}, Huisken considered the deformation of hypersurfaces of the sphere by their mean curvature,
he proved if the initial hypersurface satisfies a suitable pinching condition, then either the hypersurfaces
shrink to a round point in finite time or the equation  has a smooth solution $M_t$ for $0\leq t<\infty$  and
$M_t$ converges to a totally geodesic hypersurface when $t$ tends to $\infty$.
We can show the similar result also holds under $(1.2)$,
under the assumption that the metric $\bar{g}_0$ of $N^{n+1}$ has positive curvature operator and is sufficiently pinched.
To be precise, we prove
\begin{thm}\label{thm1.1}
There exists a positive constant $\varepsilon_0\leq\frac{1}{4(n+1)}$ small, such that if $(N^{n+1},\bar{g}_0)$ satisfies
\begin{equation}\tag{1.3}\label{1.3}
\|\bar{R}_{\alpha\beta\gamma\delta}-(\bar{g}_{\alpha\gamma}
\bar{g}_{\beta\delta}-\bar{g}_{\alpha\delta}
\bar{g}_{\beta\gamma})\|^2\leq\varepsilon_0^2,\ \
\|\bar{\nabla}\bar{R}m\|\leq \varepsilon_0
\end{equation}
where the norm $\|\cdot\|$ is taken with respect to $\bar{g}_0$, and  the initial hypersurface $M_0$ immersed into $(N^{n+1},\bar{g}_0)$ satisfies
\begin{equation}\tag{1.4}\label{1.4}
\|A\|^2\leq\alpha_nH^2+1
\end{equation}
with
\[\alpha_2=\frac{11}{16}, \ \ \alpha_n=\frac{4}{4n-3}, n\geq 3\]
then for the solution to $(1.2)$,  either

    $(1)$ $M_t$ shrink to a round point in finite time $T<\infty$,
     and $\max_{M_t}|H|\rightarrow\infty$ as $t\rightarrow T$; or

    $(2)$ the equation has a solution $M_t$ for $0\leq t< \infty$, and $M_t$ converge to a totally geodesic sphere in
$C^{\infty}$-topology.
\end{thm}

  From $(\ref{1.3})$,  we know $(N^{n+1},\bar{g}_0)$ has positive curvature,
by the result of Hamilton\cite{Ha} and Huisken\cite{Hu2},
$(N^{n+1},\bar{g}(t))$ converge to the spherical space form as $t\rightarrow \infty$.
But it is not easy to see the behaviour of the
hypersurface with its induced metric evolving under the mean curvature flow. The key problem is
when the mean curvature flow will develop singularities in a finite time.
If it will not develop a singularity, we wish to understand which one is faster between
the background manifold to the sphere under Ricci flow and the immersed hypersurface to its totally
geodesic hypersurface under mean curvature flow.

  The rest of the paper is organized as follows. In section 2,
we give some preliminary and get the evolution equations for quantities of hypersurfaces.
In section 3, we  derive a pinching estimate to
control the second fundamental form  by using an inequality derived above. In section 4,
we  show the gradient of the mean curvature can be controlled by the mean curvature itself.
We  give the proof of Theorem 1.1 in the last section .

\section{ Preliminaries and Evolution Equations}

In this section, we  gather some estimates which will be used later.
We choose a local  frames field
$\{e_0, e_1, \cdots, e_n\}$ in $N^{n+1}$ such that $e_0=\nu, e_i=\frac{\partial X}{\partial x_i}$ on $X(\cdot)$. Let $\nabla$ and $\Delta$
denote the connection and Laplacian on $M$ determined by the induced metric $g$.
We denote all the quantities on $(N^{n+1},\bar{g})$ with a bar, for example,
by  $\bar{\nabla}$ the covariant derivative, $\bar{\Delta}$ the
Laplacian,  and $\bar{R}m=\bar{R}_{\alpha\beta\gamma\delta}$ the Riemannian curvature tensor.
Let $\overset{\circ}{R}m$ be the tracefree part of curvature operator, i.e,
\[\overset{\circ}{R}_{\alpha\beta\gamma\delta}=R_{\alpha\beta\gamma\delta}
-\frac{R}{n(n+1)}(g_{\alpha\gamma} g_{\beta\delta}-g_{\alpha\delta}g_{\beta\gamma})\]
and
\[\bar{E}_{\alpha\beta\gamma\delta}=\bar{R}_{\alpha\beta\gamma\delta}-
\frac{\bar{r}}{n(n+1)}(\bar{g}_{\alpha\gamma}\bar{g}_{\beta\delta}
-\bar{g}_{\alpha\delta}\bar{g}_{\beta\gamma}) \]
We will show the exponential decay  of $\|\bar{E}\|$ and $\|\bar{\nabla}\bar{R}m\|$ under the normalized Ricci flow.
First we consider  the Ricci flow with $\tilde{g}(\cdot,0)=\bar{g}_0$,
\[\frac{\partial }{\partial \tilde{t}}\tilde{g}_{\alpha\beta}=-2\tilde{R}_{\alpha\beta}, \quad \tilde{t}\in[0,T),\]
where $T$ is the singular time, and $(N^{n+1},\bar{g}_0)$ satisfies the assumption (1.3) for some constant
$\varepsilon_0$.

By our assumption, the sectional curvature $\tilde{K}(x,0)$ and the scalar curvature $\tilde{R}(x,0)$ of $\bar{g}_0$
satisfy
\begin{equation}\label{2.1}
1-\varepsilon_0\leq \tilde{K}(x,0)\leq 1+\varepsilon_0,
\ \ n(n+1)(1-\varepsilon_0)\leq\tilde{R}(x,0)\leq n(n+1)(1+\varepsilon_0),
\end{equation}
which is followed by
\begin{align*}
\|\overset{\circ}{\tilde{R}}m\|^2(x,0)&\leq\|\tilde{R}_
{\alpha\beta\gamma\delta}-(\tilde{g}_{\alpha\gamma}\tilde{g}_{\beta\delta}
-\tilde{g}_{\alpha\delta}\tilde{g}_{\beta\gamma})\|^2(x,0)
+\|(1-\frac{\tilde{R}}{n(n+1)})(\tilde{g}_{\alpha\gamma}\tilde{g}_{\beta\delta}
-\tilde{g}_{\alpha\delta}\tilde{g}_{\beta\gamma})\|^2(x,0)\\
&\leq \varepsilon_0^2+2n(n+1)\varepsilon_0^2\leq\frac
{2(n+1)^2\varepsilon_0^2\tilde{R}^2(x,0)}{n^2(n+1)^2(1-\varepsilon_0)^2}\\
&\leq\frac{\tilde{R}^2(x,0)}{4n^2(n+1)^2}
\end{align*}
We need the following results which were derived by Huisken in \cite{Hu2} and take the following version in
our case.
\begin{lm}\label{lm 2.1}\textbf{(Theorem 3.1 of \cite{Hu2}).}
Under the assumption $(1.3)$, it always holds
\[\|\tilde{R}m\|^2-\frac{2}{n(n+1)}\tilde{R}^2\leq\frac{\tilde{R}^2}{4n^2(n+1)^2},\]
which implies the sectional curvature $\tilde{K}(x,\tilde{t})$ of $(N^{n+1},\tilde{g})$ satisfyes
$\tilde{K}(x,\tilde{t})\geq \frac{\tilde{R}(x,\tilde{t})}{2n(n+1)}$.
Moreover, there exist constants $C_0<\infty$ and $\delta_0\in(0,1)$ depending only on
$n$ such that $\|\overset{\circ}{\tilde{R}}m\|^2\leq C_0\tilde{R}^{2-\delta_0}$
holds on $0\leq \tilde{t}< T$.
\end{lm}
\begin{rem}\label{rem 2.1}
In Theorem 3.1 of \cite{Hu2}, Huisken gave the explicit  expression of $C_0$, i.e,
\[C_0=\sup_{(N^{n+1},\tilde{g}(0))}\|\overset{\circ}{\tilde{R}}m\|^2\tilde{R}^{\delta_0-2},\]
by our assumption (1.3), $C_0\leq \varepsilon_0^2$.
\end{rem}
\begin{lm}\label{lm 2.2}\textbf{(Theorem 4.1 of \cite{Hu2}).}
For any $\eta>0$, we can find $C(\eta)$ depending only  on $\eta$ and $n$,
such that on $0\leq \tilde{t}<T$ we have
\[\|\tilde{\nabla}\tilde{R}\|^2\leq \eta \tilde{R}^3+C(\eta)\]
\end{lm}

Let $V$ be the volume of $(N^{n+1},\bar{g}_0)$.
We choose the normalization factor
$\psi(\tilde{t})=(\frac{\int d\mu_{\tilde{g}}}{V})^{-\frac{2}{n+1}}$
and a new time scale $t=\int_0^{\tilde{t}}\psi(s)ds$,
then $\bar{g}(t)=\psi(\tilde{t})\tilde{g}(\tilde{t})$
satisfy the normalized Ricci flow with $\bar{g}(0)=\bar{g}_0$ and
$\frac{d \ln \psi}{dt}=\frac{2}{n+1}\bar{r}$. Define a function  $\phi$ by $\phi(t)=\psi(\tilde{t})$.
 The following evolution equations for
the normalized Ricci flow were established by Hamilton in \cite{Ha}.
\begin{lm}\label{lm 2.3}  Under the normalized Ricci flow,
\begin{align*}
&(1) \frac{\partial }{\partial t}\|\bar{R}m\|^2=
\bar{\Delta}\|\bar{R}m\|^2-2\|\bar{\nabla}\bar{R}m\|^2
+4\bar{Q}_{\alpha\beta\gamma\delta}
\bar{R}_{\alpha\beta\gamma\delta}
-\frac{4}{n+1}\bar{r}\|\bar{R}m\|^2,\\
&(2) \frac{\partial}{\partial t}\|\bar{R}ic\|^2=
\bar{\Delta}\|\bar{R}ic\|^2-2\|\bar{\nabla}\bar{R}ic\|^2
+4\bar{R}_{\alpha\beta}\bar{R}_{\gamma\delta}\bar{R}_{\alpha\gamma\beta\delta}
-\frac{4}{n+1}\bar{r}\|\bar{R}ic\|^2,\\
&(3) \frac{\partial}{\partial t}\bar{R}
=\bar{\Delta}\bar{R}+2\|\bar{R}ic\|^2
-\frac{2}{n+1}\bar{r}\bar{R}.
\end{align*}
where $\bar{Q}_{\alpha\beta\gamma\delta}=(\bar{B}_{\alpha\beta\gamma\delta}-\bar{B}_{\alpha\beta\delta\gamma}
-\bar{B}_{\alpha\delta\beta\gamma}+\bar{B}_{\alpha\gamma\beta\delta})
\bar{R}_{\alpha\beta\gamma\delta}$, and
$\bar{B}_{\alpha\beta\gamma\delta}=\bar{R}_{\alpha\eta\beta\theta}
\bar{R}_{\gamma\eta\delta\theta}$.
\end{lm}
Now we are ready to prove
\begin{thm}\label{thm 2.1}
There exist some universal constant $\bar{C}$ and  $\lambda$ depending only on $n$  such that
under the normalized Ricci flow,
\[\|\bar{E}\|(\cdot,t) \leq\bar{C}\varepsilon_0 e^{-\lambda t},\quad
\|\bar{\nabla}\bar{R}m(\cdot,t)\|\leq \bar{C}\varepsilon_0e^{-\lambda t}\]
\end{thm}
\begin{proof} By the evolution equations,
\begin{align*}
\frac{d\bar{r}}{dt}=&\frac{\int\frac{\partial\bar{R}}{\partial t}d\mu}{\int d\mu}
=2(\frac{\int\|\bar{R}ic\|^2d\mu}{\int d\mu}-\frac{\bar{r}^2}{n+1})\\
\geq&\frac{2}{n+1}(\frac{\int\bar{R}^2d\mu}{\int d\mu}-\bar{r}^2)\geq0
\end{align*}
so
\begin{equation}\label{2.2}
n(n+1)(1-\varepsilon_0)\leq\bar{r}(0)\leq \bar{r}(t), \ t\in[0,\infty)
\end{equation}
Using the upper bound for the sectional curvature of $(N^{n+1},\bar{g}_0)$ and
Klingenberg's Lemma (Theorem 5.10 of \cite{CE}), the injectivity radius $r_0$
of $(N^{n+1},\bar{g}_0)$ satisfies
$r_0\geq \frac{\pi}{\sqrt{1+\varepsilon_0}}$.
Let $\omega_{n+1}$ be the volume of unit sphere $S^{n+1}$.
Then the volume comparison theorem implies
$V\geq \omega_{n+1}(1+\varepsilon_0)^{-\frac{n+1}{2}}$.
Since $(N^{n+1},\bar{g}(t))$ converges to $(N^{n+1},\bar{g}_{\infty})$
with constant curvature $K_{\infty}$, thus
\[\omega_{n+1}(1+\varepsilon_0)^{-\frac{n+1}{2}}\leq V=
V_{\infty}=\omega_{n+1}K_{\infty}^{-\frac{n+1}{2}},\]
which follows by
\begin{equation}\label{2.3}
\bar{r}(t)\leq \bar{r}_{\infty}= n(n+1)K_{\infty}\leq n(n+1)(1+\varepsilon_0)
\end{equation}
As
\[\frac{d \ln \phi(t)}{dt}=\frac{d \ln \psi(\tilde{t})}{dt}=\frac{2}{n+1}\bar{r}(t)\geq 2,\]
we have
\[\phi(t)\geq\phi(0)e^{2t}=\psi(0)e^{2t}\]
By Lemma \ref{lm 2.2},
\[\|\bar{\nabla}\bar{R}\|\leq \eta\bar{R}^{\frac{3}{2}}+C(\eta)\phi(t)^{-\frac{3}{2}}\leq\eta\bar{R}^{\frac{3}{2}}+C(\eta)\]
\textit{Step 1. } We first show there exists a constant $C_n$ depending only on $n$ such that
for any initial metric $\bar{g}_0$ satisfying $(1.3)$, the corresponding normalized Ricci flow $(N^{n+1},\bar{g}(t))$
satisfies
\begin{equation}\label{2.4}
\bar{R}(x,t)\leq C_n, \ \forall \  (x,t)\in N^{n+1}\times[0,\infty)
\end{equation}

We show  this by a contradiction argument.
Suppose not, then there exist a sequence of metrics $\bar{g}_k$ satisfying $(1.3)$, $x_k\in N^{n+1}$ and $t_k>0$ such that
$A_k=\bar{R}_k(x_k,t_k)\rightarrow \infty$ as $k\rightarrow \infty$.
For any $\eta>0$, there exists an integer $k$, such that the metric $\bar{g}_k(t)$ satisfies
\[\|\bar{\nabla}\bar{R}\|(\cdot,t_k)\leq 2\eta A_k^{\frac{3}{2}}\]
Now for any point $y$ with $d_{\bar{g}_k(t_k)}(y,x_k)\leq \frac{1}{\sqrt{\eta A_k}}$, we have
\[\bar{R}_k(y)\geq A_k-2d_{\bar{g}_k(t_k)}(x,x_k)\eta A_k^{\frac{3}{2}}\geq(1-2\sqrt{\eta})A_k\]
Then by Lemma 2.1, the sectional curvature $\bar{K}(y,t_k)$ of $(N^{n+1},\bar{g}_k(t_k))$
satisfies
\[\bar{K}(y,t_k)\geq \frac{1-2\sqrt{\eta}}{2n(n+1)}A_k\]

   On the other hand, by Myers' theorem,
any geodesic from $x_k$ with length larger than $\frac{2(n+1)\pi}{\sqrt{(1-2\sqrt{\eta})A_k}}$
must have conjugate points. Thus by choosing $\eta< \frac{1}{8(n+1)^2\pi}$ and $k$ large enough, for any $x\in (N^{n+1},\bar{g}_k(t_k))$,  we have
\[\bar{K}(x,t_k)\geq \frac{1-2\sqrt{\eta}}{2n(n+1)}A_k.\]
Hence $\text{Vol}(N^{n+1}, \bar{g}_k({t_k}))\rightarrow 0$ as
$k \rightarrow \infty$, which contradicts with the fact that
$(N^{n+1}, \bar{g}_k(t))$ has constant volume $V\geq\omega_{n+1}(1+\varepsilon_0)^{-\frac{n+1}{2}}$.

\textit{Step 2. } We next show the exponentially decreasing of
$\|\bar{\nabla}\bar{R}m\|$ under the normalized Ricci flow.

Let $C_n$ denote the universal constants depending only on $n$. By Lemma 2.1 and $(2.4)$,
\begin{equation}\tag{2.5}\label{2.5}
\|\overset{\circ}{\bar{R}}m\|^2=\|\bar{R}m\|^2-\frac{2\bar{R}^2}{n(n+1)}
\leq C_0\bar{R}^{2-\delta_0}\phi(t)^{-\delta_0}\leq C_n\varepsilon_0^2e^{-2\delta_0t}
\end{equation}
Let $f=\|\bar{R}m\|^2-\frac{2\bar{R}^2}{n(n+1)}$.
By Lemma \ref{lm 2.3},
\begin{equation}\tag{2.6}\label{2.6}
\frac{\partial}{\partial t}f\leq \bar{\Delta}f-2\|\bar{\nabla}\bar{R}m\|^2+
\frac{4\|\bar{\nabla}\bar{R}\|^2}{n(n+1)}+ 4\bar{Q}_{\alpha\beta\gamma\delta}
\bar{R}_{\alpha\beta\gamma\delta}-\frac{8\bar{R}}{n(n+1)}\|\bar{R}ic\|^2,
\end{equation}
By
\begin{align*}
\bar{Q}_{\alpha\beta\gamma\delta}\bar{R}_{\alpha\beta\gamma\delta}\leq&
\|\bar{Q}_{\alpha\beta\gamma\delta}(\bar{R}_{\alpha\beta\gamma\delta}-
\frac{\bar{R}}{n(n+1)}(\bar{g}_{\alpha\gamma}\bar{g}_{\beta\delta}-
\bar{g}_{\alpha\delta}\bar{g}_{\beta\gamma}))\|\nonumber\\
&+\frac{\bar{R}}{n(n+1)}\bar{Q}_{\alpha\beta\gamma\delta}
(\bar{g}_{\alpha\gamma}\bar{g}_{\beta\delta}-
\bar{g}_{\alpha\delta}\bar{g}_{\beta\gamma})\tag{2.7}\label{2.7}
\end{align*}
An easy  calculation shows
\begin{equation}\tag{2.8}\label{2.8}
\bar{Q}_{\alpha\beta\gamma\delta}
(\bar{g}_{\alpha\gamma}\bar{g}_{\beta\delta}-
\bar{g}_{\alpha\delta}\bar{g}_{\beta\gamma})
=2(\|\bar{R}m\|^2+\|\bar{R}ic\|^2)-4\bar{R}_{\alpha\beta\gamma\theta}
\bar{R}_{\gamma\beta\alpha\theta}
\end{equation}
By taking $C_n$ large enough, we have
\begin{align*}
\bar{R}_{\alpha\beta\gamma\theta}
\bar{R}_{\gamma\beta\alpha\theta}\geq&
\frac{\bar{R}}{n(n+1)}(\bar{g}_{\alpha\gamma}\bar{g}_{\beta\theta}-
\bar{g}_{\alpha\theta}\bar{g}_{\beta\gamma})\bar{R}_{\gamma\beta\alpha\theta}\nonumber\\
&-\|\frac{\bar{R}}{n(n+1)}(\bar{g}_{\alpha\gamma}\bar{g}_{\beta\theta}-
\bar{g}_{\alpha\theta}\bar{g}_{\beta\gamma})-\bar{R}_{\alpha\beta\gamma\theta}\|
\|\bar{R}_{\gamma\beta\alpha\theta}\|\nonumber\\
\geq&\frac{\bar{R}^2}{n(n+1)}-C_n\varepsilon_0e^{-\delta_0t}\tag{2.9}\label{2.9}
\end{align*}
Substituting (\ref{2.8}) and (\ref{2.9}) into (\ref{2.7}) gives
\begin{equation}\tag{2.10}\label{2.10}
\bar{Q}_{\alpha\beta\gamma\delta}\bar{R}_{\alpha\beta\gamma\delta}\leq
\frac{2\bar{R}\|\bar{R}ic\|^2}{n(n+1)}+\frac{2\bar{R}}{n(n+1)}
(\|\bar{R}m\|^2-\frac{2\bar{R}^2}{n(n+1)})+
C_n\varepsilon_0e^{-\delta_0t}
\end{equation}
Combining (\ref{2.6}) and (\ref{2.10}), we get
\begin{equation}\tag{2.11}\label{2.11}
\frac{\partial}{\partial t}f\leq \bar{\Delta}f-2\|\bar{\nabla}\bar{R}m\|^2+
\frac{4\|\bar{\nabla}\bar{R}\|^2}{n(n+1)}+C_n\varepsilon_0e^{-\delta_0t}
\end{equation}
By Lemma 4.3 in \cite{Hu2},
\[\|\bar{\nabla}\bar{R}ic\|^2-\frac{\|\bar{\nabla}\bar{R}\|^2}{n+1}\geq
(\frac{3n+1}{2n(n+3)}-\frac{1}{n+1})\|\bar{\nabla}\bar{R}\|^2
=\frac{(n-1)^2}{2n(n+1)(n+3)}\|\bar{\nabla}\bar{R}\|^2
\]
Using Lemma 2.3, we have
\begin{align*}
\frac{\partial}{\partial t}(\|\bar{R}ic\|^2-\frac{\|\bar{R}\|^2}{n+1})\leq&
\bar{\Delta}(\|\bar{R}ic\|^2-\frac{\|\bar{R}\|^2}{n+1})-
2(\|\bar{\nabla}\bar{R}ic\|^2-\frac{\|\bar{\nabla}\bar{R}\|^2}{n+1})\\
&+4(\bar{R}_{\alpha\beta}-\frac{\bar{R}}{n+1}\bar{g}_{\alpha\beta})
\bar{R}_{\gamma\delta}\bar{R}_{\alpha\gamma\beta\delta}\\
\leq&\bar{\Delta}(\|\bar{R}ic\|^2-\frac{\|\bar{R}\|^2}{n+1})
-\frac{(n-1)^2}{n(n+1)(n+3)}\|\bar{\nabla}\bar{R}\|^2
+C_n\varepsilon_0^2e^{-\delta_0t},\tag{2.12}\label{2.12}
\end{align*}
where we have used the fact
\[(\bar{R}_{\alpha\beta}-\frac{\bar{R}}{n+1}\bar{g}_{\alpha\beta})
\bar{R}_{\gamma\delta}\bar{R}_{\alpha\gamma\beta\delta}
\leq\bar{R}(\|\bar{R}ic\|^2-\frac{\bar{R}^2}{n+1})\leq C_n\varepsilon_0^2e^{-\delta_0t}\]
In additional,
\begin{equation}\tag{2.13}\label{2.13}
\frac{\partial }{\partial t}\|\bar{\nabla}\bar{R}m\|^2
\leq\bar{\Delta}\|\bar{\nabla}\bar{R}m\|^2-2\|\bar{\nabla}^2\bar{R}m\|^2+
C_n\|\bar{\nabla}\bar{R}m\|^2
\end{equation}
Now let
\[F=\|\bar{\nabla}\bar{R}m\|^2+C_nf+C_n^2(\|\bar{R}ic\|^2-\frac{\|\bar{R}\|^2}{n+1})
\]
Combining (\ref{2.11}), (\ref{2.12}) and (\ref{2.13}) gives
\[\frac{\partial F}{\partial t}\leq\bar{\Delta}F-C_nF+2C_n^3\varepsilon_0e^{-\delta_0 t} \]
Since $F(\cdot,0)\leq 3C_n^3\varepsilon_0^2$,  the standard maximum principle implies that
 there exists a constant $\lambda_1$ depending only on $n$
such that
\begin{equation}\tag{2.14}\label{2.14}
\|\bar{\nabla}\bar{R}m\|^2\leq C\varepsilon_0^2e^{-2\lambda_1t}.
\end{equation}

\textit{Step 3.}
We want to get an uniformly upper bound for the diameter of $(N^{n+1},\bar{g}(t))$ under the normalized Ricci flow.

Consider Perelman's $\mathcal{W}$-functional \cite{Pe},
\[\mathcal{W}(\tilde{g},f,\tau)=\int_{N^{n+1}}
[\tau(\tilde{R}+\|\tilde{\nabla}f\|^2)+f-(n+1)](4\pi\tau)^{-\frac{n+1}{2}}e^{-f}d\mu\]
where $f$ is a smooth function on $N^{n+1}$, and $\tau$ is a positive scale parameter.
Let
\[\rho=(4\pi\tau)^{-\frac{n+1}{4}}e^{-\frac{f}{2}}\]
Now we set
\[\mu(\tilde{g},\tau)=\inf\{\mathcal{W}(\tilde{g},f,\tau)
\|\rho\in C^{\infty}(N^{n+1}),\int_{N^{n+1}}\rho^2d\mu=1\}\]
By our assumption for the initial metric $\bar{g}_0$ and the Theorem A in \cite{Ye},
\[\mu(\tilde{g}(0),\tau)\geq -CT-C,\ \tau\in(0,2T]\]
where $C$ is a constant depending only on $n$, and $T<\frac{n+1}{2\tilde{R}_{\min}(0)}$
is the maximal existence time  for the (unnormalized) Ricci flow.
Thus
\[\mu(\tilde{g}(0),\tau)\geq -C_n,\  \forall \tau\in(0, 2T]\]
Now combining the upper bound for the scalar curvature of $\bar{g}(\cdot,t)$,
Perelman's no local collapsing theorem $I'$ \cite{Pe}, and a local injectivity radius estimate of
Cheeger-Gromov-Taylor \cite{CGT}, we can get the following
\begin{prop}
There exists a constant $c_n> 0$ depending on $n$, such that
\[\text{inj}(N^{n+1}, \bar{g}(t))\geq c_n,\ \forall\  t\in[0,\infty)\]
\end{prop}
As $(N^{n+1}, \bar{g}(t))$ has constant volume, it follows that diameter of $(N^{n+1}, \bar{g}(t))$
has a uniformly upper bound
\begin{equation}\tag{2.15}\label{2.15}
\text{diam}(N^{n+1},\bar{g}(t))\leq C_n, \ \forall \  t\in[0,\infty)
\end{equation}
Combining (\ref{2.5}), (\ref{2.14})  and (\ref{2.15}) yields the desired estimate.
\end{proof}
\begin{rem}\label{rem 2.2}
Once getting the exponential decay of $\|\overset{\circ}{\bar{R}}m\|^2$ and
$\|\bar{\nabla}\bar{R}m\|^2$, one can show $\|\bar{\nabla}^k\bar{R}m\|^2$
are also exponentially decreasing, see \cite{Ha} for details.
\end{rem}
\begin{rem}\label{rem 2.3}
Note that from (\ref{2.2}) and (\ref{2.3}), we have derived the uniform bound for $\bar{r}(t)$,
\begin{equation}\tag{2.16}\label{2.16}
n(n+1)(1-\varepsilon_0)\leq\bar{r}(t)\leq n(n+1)(1+\varepsilon_0),\ \ t\in[0,\infty)
\end{equation}
\end{rem}

We denote by $S=\{S_i\}$ the vector with components $S_i=\bar{R}_{0i}$,
the following estimate  was derived by Huisken in \cite{Hu3}.
\begin{lm}\label{lm 2.4}
For any $\eta>0$,
\[\|\nabla A\|^2\geq(\frac{3}{n+2}-\eta)\|\nabla H\|^2-\frac{2}{n+2}
(\frac{2}{n+2}\eta^{-1}-\frac{n}{n-1})\|S\|^2\]
\end{lm}
\quad By a direct calculation or using the results in \cite{Lo}, we could establish the following evolution equations
\begin{lm}\label{lm 2.5}
\begin{align*}
&(1)\frac{\partial g_{ij}}{\partial t}=-2Hh_{ij}-2\bar{R}_{ij}+\frac{2}{n+1}\bar{r}g_{ij},\\
&(2)\frac{\partial h_{ij}}{\partial t}=\Delta
h_{ij}-2Hh_{ip}h_{jp}+\left\vert A\right\vert ^{2}h_{ij}+\frac{\bar{r}}{n+1}h_{ij}+P_{ij}
-\bar{
\nabla}_{0}\bar{R}_{0i0j},\\
&(3)\frac{\partial H}{\partial t}=\Delta H+\|A\|^{2}H+2\bar{R}_{ij}h_{ij}
-\frac{\bar{r}H}{n+1}-\bar{\nabla}_{0}\bar{R}_{00},\\
&(4)\frac{\partial\| A\|^{2}}{\partial t}=\Delta
\|A\|^2-2\|\nabla A\|^2+2\|A\|^{4}+2P_{ij}h_{ij}
+4\bar{R}_{ij}h_{ik}h_{jk}-\frac{2\bar{r}\left\vert A\right\vert ^{2}}{n+1}
-2\bar{\nabla}_{0}\bar{R}_{0i0j}h_{ij}.
\end{align*}
\text{Here}$\ P_{ij}=2h_{kl}\bar{R}_{kilj}-h_{il}\bar{R}_{jklk}-h_{jl}\bar{R}_{iklk}$.
\end{lm}

For simplicity, we will use  the following denotation throughout the paper,
\[u=2\bar{R}_{ij}h_{ij}-\frac{\bar{r}H}{n+1}-\bar{\nabla}_0\bar{R}_{00}\]
\[v=2P_{ij}h_{ij}+4\bar{R}_{ij}h_{ik}h_{jk}-\frac{2\bar{r}\|A\|^2}{n+1}-2\bar{\nabla}_0\bar{R}_{0i0j}h_{ij}\]
Now we choose $\varepsilon_0$ small, such that
\begin{equation}\tag{2.17}\label{2.17}
\varepsilon_1=(\bar{C}+1)\varepsilon_0\leq\frac{1}{2^7n}
\end{equation}
By
\begin{equation}\tag{2.18}\label{2.18}
\|\bar{E}\|=\|\bar{R}_{\alpha\beta\gamma\delta}-
\frac{\bar{r}}{n(n+1)}(\bar{g}_{\alpha\gamma}\bar{g}_{\beta\delta}
-\bar{g}_{\alpha\delta}\bar{g}_{\beta\gamma})\|
\leq \bar{C}\varepsilon_0 e^{-\lambda t},
\end{equation}
it follows the sectional curvature $\bar{K}(x,t)$
satisfies
\[\frac{\bar{r}}{n(n+1)}-\bar{C}\varepsilon_0 e^{-\lambda t}
\leq \bar{K}(x,t)\leq \frac{\bar{r}}{n(n+1)}+\bar{C}\varepsilon_0 e^{-\lambda t}\]
Taking the trace on $\beta$ and $\delta$ in  (\ref{2.18})  gives
\begin{equation}\tag{2.19}\label{2.19}
\|\bar{R}_{ij}-\frac{\bar{r}}{n+1}\bar{g}_{ij}\|\leq (n+1)\bar{C}\varepsilon_0 e^{-\lambda t}
\end{equation}
At any point $x\in M_t$, we choose an orthonormal
basis $\{e_1,\cdots,e_n\}$ such that $g_{ij}=\delta_{ij}$, $h_{ij}=\kappa_i\delta_{ij}$,
then
\begin{align*}
P_{ij}h_{ij}=\sum_{i,p}2(\kappa_i\kappa_p-\kappa_i^2)\overline{R}_{pipi}
=-\sum_{i<p}2(\kappa_i-\kappa_p)^2\overline{R}_{pipi},
\end{align*}
and
\begin{equation}\tag{2.20}\label{2.20}
-2n\bar{C}\varepsilon_0e^{-\lambda t}\|A\|^2
\leq P_{ij}h_{ij}-\frac{2\bar{r}}{n+1}(\|A\|^2-\frac{H^2}{n})
\leq2n\bar{C}\varepsilon_0e^{-\lambda t}\|A\|^2
\end{equation}
By
\[
-\|(\bar{R}_{ij}-\frac{\bar{r}}{n+1}\bar{g}_{ij})h_{ik}h_{jk}\|
\leq \bar{R}_{ij}h_{ik}h_{jk}-\frac{\bar{r}}{n+1}\bar{g}_{ij}h_{ik}h_{jk}
\leq \|(\bar{R}_{ij}-\frac{\bar{r}}{n+1}\bar{g}_{ij})h_{ik}h_{jk}\|\]
we get
\begin{equation}\tag{2.21}\label{2.21}
-(n+1)\bar{C}\varepsilon_0 e^{-\lambda t}\|A\|^2
\leq \bar{R}_{ij}h_{ik}h_{jk}-\frac{\bar{r}}{n+1}\|A\|^2
\leq (n+1)\bar{C}\varepsilon_0 e^{-\lambda t}\|A\|^2
\end{equation}
Now we have
\begin{align*}
v\leq&\frac{-4\bar{r}}{n+1}(\|A\|^2-\frac{H^2}{n})+4n\bar{C}\varepsilon_0e^{-\lambda t}\|A\|^2
+\frac{4\bar{r}}{n+1}\|A\|^2+4(n+1)\bar{C}\varepsilon_0e^{-\lambda t}\|A\|^2\\
&-\frac{2\bar{r}}{n+1}\|A\|^2+2n\bar{C}\varepsilon_0e^{-\lambda t}\|A\|\\
\leq&\frac{-2\bar{r}}{n+1}\|A\|^2+\frac{4\bar{r}}{n(n+1)}H^2+(8n+4)\bar{C}\varepsilon_0\|A\|^2
+2n\bar{C}\varepsilon_0\|A\|^2\\
\leq&-2n(1-\varepsilon_0)\|A\|^2+4(1+\varepsilon_0)H^2+(8n+4)\bar{C}\varepsilon_0\|A\|^2\\
\leq&-2n\|A\|^2+4H^2+(2n+8n\bar{C}+4\bar{C})\varepsilon_0\|A\|^2+2n\bar{C}\varepsilon_0\|A\|\\
\leq&-2n\|A\|^2+4H^2+10n\varepsilon_1\|A\|^2+2n\varepsilon_1\|A\|\tag{2.22}\label{2.22}
\end{align*}
where we have used (\ref{2.16}) in the third inequality and (\ref{2.17}) in the last inequality.
Similarly, using
\[-2\|H(\bar{R}_{ij}-\frac{\bar{r}}{n+1}\bar{g}_{ij})h_{ij}\|
\leq2H(\bar{R}_{ij}-\frac{\bar{r}}{n+1}\bar{g}_{ij})h_{ij}
\leq2\|H(\bar{R}_{ij}-\frac{\bar{r}}{n+1}\bar{g}_{ij})h_{ij}\|\]
we have
\begin{equation}\tag{2.23}\label{2.23}
-2n(n+1)\bar{C}\varepsilon_0e^{-\lambda t}\|A\|^2\leq2H\bar{R}_{ij}h_{ij}
-\frac{2\bar{r}H^2}{n+1}\leq
2n(n+1)\bar{C}\varepsilon_0e^{-\lambda t}\|A\|^2
\end{equation}
Now it follows
\begin{align*}
uH\geq&\frac{\bar{r}}{n+1}H^2-2n(n+1)\bar{C}\varepsilon_0\|A\|^2-n^2\bar{C}\varepsilon_0\|A\|^2\\
\geq&nH^2-n\varepsilon_0H^2-2n(n+1)\bar{C}\varepsilon_0\|A\|^2-n^2\bar{C}\varepsilon_0\|A\|\\
\geq&nH^2-n^2\varepsilon_0\|A\|^2-2n(n+1)\bar{C}\varepsilon_0\|A\|^2-n^2\bar{C}\varepsilon_0\|A\|\\
\geq&nH^2-3n^2\varepsilon_1\|A\|^2-n^2\varepsilon_1\|A\|\tag{2.24}\label{2.24}
\end{align*}
\begin{lm}\label{lm 2.6}
Inequality (\ref{1.4}) is preserved under equation (1.2) for all times $0\leq t< T$, where $T$
is the maximal existence time of the solution to equation (1.2).
\end{lm}
\begin{proof}  From Lemma \ref{lm 2.5}, we get
\begin{align*}
\frac{\partial}{\partial t}(\|A\|^2-\alpha_nH^2-1)
=&\Delta(\|A\|^2-\alpha_nH^2-1)-2(\|\nabla A\|^2-\alpha_n\|\nabla H\|^2)\nonumber\\
&+2\|A\|^2(\|A\|^2-\alpha_nH^2)+v-2\alpha_nuH\tag{2.25}\label{2.25}
\end{align*}
Combining (\ref{2.22}) and (\ref{2.24}) gives
\begin{equation}\tag{2.26}\label{2.26}
v-2\alpha_n uH\leq(22n\varepsilon_1-2n)\|A\|^2+(4-2n\alpha_n)H^2+6n\varepsilon_1\|A\|
\end{equation}
By taking $\eta=\frac{1}{2^5}$ for $n=2$ and $\eta=\frac{1}{8(n+2)}$ for $n\geq 3$ in Lemma 2.4, we have
\begin{equation}\tag{2.27}\label{2.27}
\|\nabla A\|^2\geq\alpha_n\|\nabla H\|^2-2^5n^2\varepsilon_1^2
\end{equation}
By substituting (\ref{2.26}), (\ref{2.27}) into (\ref{2.25}) , we get
\begin{align*}
\frac{\partial}{\partial t}(\|A\|^2-\alpha_nH^2-1)\leq&\Delta(\|A\|^2-\alpha_nH^2-1)
+2\|A\|^2(\|A\|^2-\alpha_nH^2-1)\\
&+(2+22n\varepsilon_1-2n)\|A\|^2+(4-2n\alpha_nH^2)+6n\varepsilon_1\|A\|+2^6n^2\varepsilon_1^2
\end{align*}
By $6n\varepsilon_1\|A\|\leq2n\varepsilon_1\|A\|^2+18n\varepsilon_1^2$ and the definition of $\alpha_n$,
 a direct computation shows
\begin{align*}
&(2+22n\varepsilon_1-2n)\|A\|^2+(4-2n\alpha_nH^2)+6n\varepsilon_1\|A\|+2^6n^2\varepsilon_1^2\\
<&-2(n-1-12\varepsilon_1)(\|A\|^2-\alpha_nH^2-1)
\end{align*}
where we have used $\varepsilon_1\leq\frac{1}{2^7n}$. Hence
\[\frac{\partial}{\partial t}(\|A\|^2-\alpha_nH^2-1)<\Delta(\|A\|^2-\alpha_nH^2-1)+2(\|A\|^2+1+12\varepsilon_1-n)
(\|A\|^2-\alpha_nH^2-1)\]
By the maximum principle, we get the desired inequality.
\end{proof}

\section{A Pinching estimate}
In this section we want to show how the eigenvalues of the second fundamental form
close to each other when the time becomes large or  the mean curvature blows up.

\begin{thm}\label{thm3.1}
There exist constants $C_1, \sigma$ and $\delta_1$ depending on $M_0$ and  $n$ such that it always holds
\[\|A\|^2-\frac{H^2}{n}\leq C_1(H^2+1)^{1-\sigma}e^{-\delta_1 t},\]
where $\sigma\in(0,1),t\in[0,T).$
\end{thm}
\begin{proof}
For convenience, let
\[W=aH^2+1,\qquad  f_{\sigma}=\frac{\|A\|^2-\frac{H^2}{n}}{W^{1-\sigma}},\]
where $a=\alpha_n-\frac{1}{n}$.
We use $C$ to denote the constant only depending on $n$ which may vary from line to line.
By Lemma \ref{lm 2.6}, $f_0\leq 1$.
From Lemma \ref{lm 2.5}, we can  get the evolution equation of $f_0$,
\begin{align*}
\frac{\partial}{\partial t}f_0=&\frac{1}{W}
\frac{\partial}{\partial t}(\|A\|^2-\frac{H^2}{n})+(\|A\|^2-\frac{H^2}{n})
\frac{\partial}{\partial t}(\frac{1}{W})\\
=&\frac{1}{W}\{\Delta(\|A\|^2-\frac{H^2}{n})+(\frac{2}{n}\|\nabla H\|^2-2\|\nabla A\|^2)
+2(\|A\|^4-\frac{H^2}{n}\|A\|^2)+v-\frac{2uH}{n}\}\\
&-(\|A\|^2-\frac{H^2}{n})\frac{2a H}{W^2}(\Delta H+\|A\|^2H+u)
\end{align*}
Using
\begin{equation}\label{3.1}
\Delta f_0=\frac{1}{W}\{\Delta(\|A\|^2-\frac{H^2}{n})-f_0\Delta(aH^2)\}-\frac{4aH}{W}\nabla_iH\nabla_if_0,
\end{equation}
we find
\begin{align*}
\frac{\partial}{\partial t}f_0=&\Delta f_0+\frac{4a H}{W}\nabla_i H\nabla_i f_0+
\frac{2}{W}(a f_0\|\nabla H\|^2+\frac{\|\nabla H\|^2}{n}-\|\nabla A\|^2)\\
&+2\|A\|^2f_0+\frac{1}{W}(v-\frac{2}{n}uH)-\frac{2a Hf_0}{W}(\|A\|^2H+u)\tag{3.2}\label{3.2}
\end{align*}
By taking $\eta$ small enough in Lemma \ref{lm 2.4}, we have
\begin{equation}\tag{3.3}\label{3.3}
(af_0+\frac{1}{n})\|\nabla H\|^2-\|\nabla A\|^2\leq\alpha_n\|\nabla H\|^2-\|\nabla A\|^2\leq
-\frac{\|\nabla H\|^2}{2^4n}+Ce^{-\lambda t}
\end{equation}
Using (\ref{2.20}), (\ref{2.21}) and (\ref{2.23}), we get
\begin{align*}
v\leq&\frac{-4\bar{r}}{n+1}(\|A\|^2-\frac{H^2}{n})
+\frac{2\bar{r}\|A\|^2}{n+1}+Ce^{-\lambda t}(\|A\|^2+1)\\
\leq&\frac{4\bar{r}H^2}{n(n+1)}-\frac{2\bar{r}\|A\|^2}{n+1}
+Ce^{-\lambda t}(\|A\|^2+1)
\end{align*}
and
\begin{equation}\tag{3.4}\label{3.4}
\frac{\bar{r}H^2}{n+1}-Ce^{-\lambda t}(\|A\|^2+1)\leq uH
\leq\frac{\bar{r}H^2}{n+1}+Ce^{-\lambda t}(\|A\|^2+1)
\end{equation}
Now it follows
\begin{equation}\tag{3.5}\label{3.5}
v-\frac{2}{n}Hu
\leq -\frac{2\bar{r}}{n+1}(\|A\|^2-\frac{H^2}{n})+Ce^{-\lambda t}(\|A\|^2+1),
\end{equation}
so we can derive
\begin{align*}
&2\|A\|^2f_0+\frac{1}{W}(v-\frac{2}{n}uH)-\frac{2a Hf_0}{W}(\|A\|^2H+u)\\
\leq& 2f_0\|A\|^2+\frac{1}{W}\{\frac{-2\bar{r}}{n+1}(\|A\|^2-\frac{H^2}{n})
+C(\|A\|^2+1)e^{-\lambda t}\}\\
&-\frac{2af_0}{W}
\{\|A\|^2H^2+\frac{\bar{r}H^2}{n+1}-Ce^{-\lambda t}(\|A\|^2+1)\}\\
\leq&2f_0\{\|A\|^2-\frac{2\bar{r}}{n+1}-\frac{a\|A\|^2H^2}{W}-\frac{a\bar{r}H^2}{(n+1)W}\}
+Ce^{-\lambda t}\\
\leq&\frac{2f_0}{W}\{a\|A\|^2H^2+\|A\|^2-\frac{\bar{r}}{n+1}(aH^2+1)-aH^2\|A\|^2-\frac{a\bar{r}H^2}{n+1}\}
+Ce^{-\lambda t}\\
\leq&\frac{2f_0}{W}\{\alpha_n H^2+1-n(1-\varepsilon_0)(aH^2+1)-an(1-\varepsilon_0)H^2\}+Ce^{-\lambda t}\\
\leq&\frac{2f_0}{W}\{[(\alpha_n-2an(1-\varepsilon_0)]H^2-n(1-a)(1-\varepsilon_0)\}+Ce^{-\lambda t}\\
\leq&-\frac{f_0}{2}+Ce^{-\lambda t}\tag{3.6}\label{3.6}
\end{align*}
where we have used Lemma \ref{lm 2.6} and (\ref{2.16}).
Substiting (\ref{3.3}) and (\ref{3.6}) into (\ref{3.2}), we have
\begin{equation}\tag{3.7}\label{3.7}
\frac{\partial }{\partial t}f_0\leq\Delta f_0+\frac{4aH}{W}
\nabla_iH\nabla_if_0-\frac{\|\nabla H\|^2}{8nW}-\frac{1}{2}f_0+Ce^{-\lambda t}
\end{equation}
Similarly, we have
\begin{align*}
\frac{\partial}{\partial t}W^{\sigma}=&\Delta W^{\sigma}-4\sigma(\sigma-1)a^2 H^2W^{\sigma-2}\|\nabla H\|^2
-2a\sigma W^{\sigma-1}\|\nabla H\|^2\nonumber\\
&+2a\sigma H^2W^{\sigma-1}\|A\|^2+2a\sigma uHW^{\sigma-1}\tag{3.8}\label{3.8}
\end{align*}
Combining (\ref{3.7}) and (\ref{3.8}) gives
\begin{align*}
\frac{\partial}{\partial t}f_{\sigma}=&\frac{\partial}{\partial t}(f_0W^{\sigma})\leq
\Delta f_{\sigma}-2\nabla_if_0\nabla_iW^{\sigma}+ 4aHW^{\sigma-1}
\nabla_if_{0}\nabla_i H\\
&-\frac{1}{8n}W^{\sigma-1}\|\nabla H\|^2-\frac{1}{2} f_{\sigma}+4a^2\sigma(1-\sigma)H^2W^{\sigma-2}f_0\|\nabla H\|^2\\
&
-2a\sigma W^{\sigma-1}f_0\|\nabla H\|^2+2a\sigma H^2W^{\sigma-1}f_0(\|A\|^2+n)
+Ce^{-\lambda t}W^{\sigma}\tag{3.9}\label{3.9}
\end{align*}
Using
\begin{align}
\nabla_if_0\nabla_iW^{\sigma}&=2a\sigma HW^{\sigma-1}\nabla_if_0\nabla_iH\nonumber\\
\nabla_if_{\sigma}\nabla_iH&=W^{\sigma}\nabla_if_0\nabla_iH+2a\sigma W^{\sigma-1}Hf_0\|\nabla H\|^2\tag{3.10}\label{3.10}
\end{align}
we find
\begin{align*}
&-2\nabla_if_0\nabla_iW^{\sigma}+4aHW^{\sigma-1}
\nabla_if_0\nabla_i H\\
=&4a(1-\sigma)HW^{-1}\nabla_if_{\sigma}\nabla_iH-8a^2\sigma(1-\sigma)W^{\sigma-2}H^2f_0\|\nabla H\|^2\tag{3.11}\label{3.11}
\end{align*}
Substitute (\ref{3.11}) into (\ref{3.9}), we obtain
\begin{align*}
\frac{\partial}{\partial t}f_{\sigma}\leq&\Delta f_{\sigma}-4a(1-\sigma)W^{-1}H\nabla_if_{\sigma}\nabla_iH
-\frac{1}{8n}W^{\sigma-1}\|\nabla H\|^2\\
&-\frac{1}{2} f_{\sigma}+2\sigma f_{\sigma}(\|A\|^2+n)
+Ce^{-\lambda t}W^{\sigma}\tag{3.12}\label{3.12}
\end{align*}

  We can't get the desired estimate  by using the maximum principle directly due to the appearance of
$2\sigma\|A\|^2f_{\sigma}$ on the right hand of (\ref{3.12}). To proceed further, we may employ the
De Giorgi-Moser iteration, see a similar argument in \cite{Hu3}. First, we will show the sectional
curvature of $M_t$ is positive. For any point $x\in M_t$, we choose
an orthonormal basis $\{e_1,\cdot\cdot\cdot, e_n\}$ for $T_xM_{t}$, such that $h_{ij}=\kappa_i\delta_{ij}$.
\begin{lm}\label{lm 3.2}
Let $K_x(e_i,e_j)$ denote the  sectional curvature of $2$-plane $span \{e_i, e_j\}\subset T_x(M_t)$. Then
\[K_x(e_i,e_j)\geq\frac{H^2+1}{8n^2}\]
as long as (\ref{1.4}) holds.
\end{lm}
\begin{proof}[Proof of Lemma 3.1]
For any $i\neq j$,
\[\|A\|^2-\frac{H^2}{n-1}\geq-2\kappa_i\kappa_j\]
By Gauss equation,
\begin{align*}
K_x(e_i,e_j)&=\frac{1}{2}(2\bar{R}_{ijij}+2\kappa_i\kappa_j)\\
&\geq\frac{1}{2}(2-4\varepsilon_1+\frac{H^2}{n-1}-\|A\|^2)\\
&\geq\frac{1}{2}(2-4\varepsilon_1+\frac{H^2}{n-1}-\alpha_nH^2-1)\\
&\geq \frac{H^2+1}{8n^2}
\end{align*}
Since $i\neq j$ is arbitrary, we get the desired estimate.
\end{proof}

Recall Simon's identity \cite{Si},
\begin{align*}
\Delta\|A\|^2=&2h_{ij}\nabla_i\nabla_jH+2\|\nabla A\|^2+2Z+2Hh_{ij}\bar{R}_{0i0j}-2\bar{R}_{00}\|A\|^2\\
&+4\bar{R}_{kikp}h_{pj}h_{ij}-4\bar{R}_{kipj}h_{kp}h_{ij}
+2\bar{\nabla}_k\bar{R}_{0ijk}h_{ij}+2\bar{\nabla}_i\bar{R}_{0j}h_{ij},
\end{align*}
where $Z=H tr(A^3)-\|A\|^4$.
By a direct computation,
\begin{align*}
&2Z+2Hh_{ij}\bar{R}_{0i0j}-2\bar{R}_{00}\|A\|^2
+4\bar{R}_{kikp}h_{pj}h_{ij}-4\bar{R}_{kipj}h_{kp}h_{ij}
+2\bar{\nabla}_k\bar{R}_{0ijk}h_{ij}+2\bar{\nabla}_i\bar{R}_{0j}h_{ij}\\
\geq&2(\sum_{i=1}^n\kappa_i)(\sum_{i=1}^n\kappa_i^3)-2(\sum_{i=1}^n\kappa_i^2)^2
+\frac{2\bar{r}(n\|A\|^2-H^2)}{n(n+1)}-Ce^{-\lambda t}(\|A\|^2+1)\\
\geq&2\sum_{i<j}\kappa_i\kappa_j(\kappa_i-\kappa_j)^2+\frac{2\bar{r}}{n(n+1)}\sum_{i<j}(\kappa_i-\kappa_j)^2
-Ce^{-\lambda t}(\|A\|^2+1)\\
\geq&2\sum_{i<j}K_x(e_i,e_j)(\kappa_i-\kappa_j)^2-Ce^{-\lambda t}(\|A\|^2+1)\\
\geq&\frac{W}{4n}(\|A\|^2-\frac{H^2}{n})-Ce^{-\lambda t}(\|A\|^2+1)
\end{align*}
Now we have
\begin{equation}\tag{3.13}\label{3.13}
\Delta\|A\|^2\geq2h_{ij}\nabla_i\nabla_jH+2\|\nabla A\|^2+
\frac{W}{4n}(\|A\|^2-\frac{H^2}{n})-Ce^{-\lambda t}(\|A\|^2+1)
\end{equation}
Substituting the inequality above into (\ref{3.1}) gives
\begin{align*}
\Delta f_0\geq&W^{-1}\{2h_{ij}\nabla_i\nabla_jH+2\|\nabla A\|^2+
\frac{W}{4n}(\|A\|^2-\frac{H^2}{n})-Ce^{-\lambda t}(\|A\|^2+1)\\
&-\frac{2}{n}H\Delta H-\frac{2}{n}\|\nabla H\|^2-2a f_0H\Delta H-2a f_0\|\nabla H\|^2\}
-4a HW^{-1}\nabla_iH\nabla_if_0
\end{align*}
We denote by $h^0_{ij}=h_{ij}-\frac{1}{n}Hg_{ij}$ the tracefree second fundamental form.
Notice
\[\frac{2}{n}\|\nabla H\|^2+2a f_0\|\nabla H\|^2\leq2\alpha_n\|\nabla H\|^2\leq2\|\nabla A\|^2+Ce^{-\lambda t}\]
Then we derive
\begin{align*}
\Delta f_0\geq &W^{-1}\{2h^0_{ij}\nabla_i\nabla_jH+
\frac{W}{4n}(\|A\|^2-\frac{H^2}{n})-Ce^{-\lambda t}(\|A\|^2+1)\\
&-2a f_0H\Delta H
-4aH\nabla_iH\nabla_if_0\}\tag{3.14}\label{3.14}
\end{align*}
Multiplying two sides of (\ref{3.14}) by $W^{\sigma}$ yields
\begin{align*}
\Delta f_{\sigma}=&W^{\sigma}\Delta f_0+f_0\Delta W_{\sigma}+2\nabla_if_0\nabla_iW^{\sigma}\\
\geq&W^{\sigma-1}\{2h^0_{ij}\nabla_i\nabla_jH+\frac{1}{4n}W^{2-\sigma}f_{\sigma}-2af_0H\Delta H-4aH\nabla_iH\nabla_if_0\}\\
&+f_0\{\sigma W^{\sigma-1}(2aH\Delta H+2a\|\nabla H\|^2)+4a^2\sigma (\sigma-1)H^2W^{\sigma-2}\|\nabla H\|^2\}\\
&+2\nabla_if_0\nabla_iW^{\sigma}-Ce^{-\lambda t}(\|A\|^2+1)W^{\sigma-1}
\end{align*}
From (\ref{3.11}),
we have
\begin{align*}
&-4aHW^{\sigma-1}\nabla_iH\nabla_if_0+4a^2\sigma(\sigma-1)f_0W^{\sigma-1}\|\nabla H\|^2
+2\nabla_if_0\nabla_iW^{\sigma}\\
=&-4a(1-\sigma)HW^{-1}\nabla_iH\nabla_if_{\sigma}+4a^2\sigma(1-\sigma)f_0H^2W^{\sigma-2}\|\nabla H\|^2
\end{align*}
Since $\sigma<1$, we get
\begin{align*}
\Delta f_{\sigma}\geq&2W^{\sigma-1}h^0_{ij}\nabla_i\nabla_jH+\frac{W}{4n}f_{\sigma}
-2a(1-\sigma)HW^{-1}f_{\sigma}\Delta H\\
&-4a(1-\sigma)HW^{-1}\nabla_iH\nabla_if_{\sigma}
-Ce^{-\lambda t}(\|A\|^2+1)W^{\sigma-1}
\end{align*}
By multiplying this inequality by $f_{\sigma}^{p-1}$ and integrating on $M_t$, it follows
\begin{align}
\frac{1}{4n}\int W f_{\sigma}^pd\mu\leq&-(p-1)\int f_{\sigma}^{p-2}\|\nabla f_{\sigma}\|^2d\mu
-2\int W^{\sigma-1}h^0_{ij}\nabla_i\nabla_jH f_{\sigma}^{p-1}d\mu\nonumber\\
&+2a(1-\sigma)\int HW^{-1}f_{\sigma}^p\Delta Hd\mu
+4a(1-\sigma)\int HW^{-1}f_{\sigma}^{p-1}\nabla_iH\nabla_if_{\sigma}d\mu\nonumber\\
&+Ce^{-\lambda t}\int (\|A\|^2+1)W^{\sigma-1}f_{\sigma}^{p-1}d\mu\tag{3.15}\label{3.15}
\end{align}
By Codazzi equation,
\[\nabla_ih_{ij}^0=\frac{n-1}{n}\nabla_jH+\bar{R}_{0j}\]
Using Stokes' theorem, then
\begin{align*}
&2\int W^{\sigma-1}h^0_{ij}\nabla_i\nabla_jH f_{\sigma}^{p-1}d\mu\\
\geq&-2\int W^{\sigma-1}\{(p-1)\|\nabla H\|\|h_{ij}^0\|f_{\sigma}^{p-2}
+\frac{n-1}{n}\|\nabla H\|^2f_{\sigma}^{p-1}+e^{-\lambda t}\|\nabla H\|f_{\sigma}^{p-1}\}d\mu\nonumber\\
&-\int 4a(1-\sigma)W^{\sigma-2}\|h^0_{ij}\|\|H\|\|\nabla H\|^2f_{\sigma}^{p-1}d\mu\tag{3.16}\label{3.16}
\end{align*}
and
\begin{align*}
&2a(1-\sigma)\int HW^{-1}f_{\sigma}^p\Delta Hd\mu\nonumber\\
\geq & -2a\int \|\nabla H\|^2W^{-1}f_{\sigma}^p\nonumber\\
&-2a\int(2a H^2W^{-2}f_{\sigma}^p\|\nabla H\|^2+p\|H\|W^{-1}f_{\sigma}^{p-1}
\|\nabla H\|\|\nabla f_{\sigma}\|)d\mu\tag{3.17}\label{3.17}
\end{align*}
Combining (\ref{3.15}), (\ref{3.16}) and (\ref{3.17}),  we obtain
\begin{align*}
&\frac{1}{4n}\int W f_{\sigma}^pd\mu\\
\leq &-(p-1)\int f_{\sigma}^{p-2}\|\nabla f_{\sigma}\|^2d\mu+
2(p-1)\int W^{\sigma-1}\|\nabla H\|\|h_{ij}^0\|\|\nabla f_{\sigma}\|f_{\sigma}^{p-2}d\mu\\
&+\int \|\nabla H\|^2(W^{\sigma-1}f_{\sigma}^{p-1}+2aW^{-1}f_{\sigma}^p)d\mu
+4\int W^{\sigma-1}e^{-\lambda t}\|\nabla H\|f_{\sigma}^{p-1}d\mu\\
&+\int 4a(1-\sigma)W^{\sigma-2}\|h^0_{ij}\|\|H\|\|\nabla H\|^2f_{\sigma}^{p-1}d\mu
+4a^2\int H^2W^{-2}f_{\sigma}^p\|\nabla H\|^2d\mu\\
&+2a(2+p)\int\|H\|W^{-1}f_{\sigma}^{p-1}
\|\nabla H\|\|\nabla f_{\sigma}\|d\mu+Ce^{-\lambda t}\int (\|A\|^2+1)W^{\sigma-1}f_{\sigma}^{p-1}d\mu
\end{align*}
Using
\[\|h_{ij}^0\|^2=\|A\|^2-\frac{H^2}{n}=f_{\sigma}W^{1-\sigma},\ \
 \|a H\|\leq W^{\frac{1}{2}},\ \ \ \ \ \ f_{\sigma}\leq W^{\sigma}
\]
and Cauchy-Schwartz inequality, we derive
\begin{lm}\label{lm 3.2}
Let $p\geq 2$. Then for any $\theta>0$ and any $\sigma\in[0,\frac{1}{4}]$, it holds
\begin{align*}
\frac{1}{4n}\int Wf_{\sigma}^pd\mu\leq&
(2\theta(p+1)+8)\int W^{\sigma-1}f_{\sigma}^{p-1}\|\nabla H\|^2d\mu\\
&+\frac{2p}{\theta}\int f_{\sigma}^{p-2}\|\nabla f_{\sigma}\|^2d\mu
+Ce^{-\lambda t}\int W^{\sigma}f_{\sigma}^{p-1}d\mu
\end{align*}
\end{lm}

Now we are ready to give an estimate for $L^p$-norm of $f_{\sigma}$, if $\sigma$ is of order $p^{-\frac{1}{2}}$.

\begin{lm}\label{lm 3.3}
For any $p\geq2^6n^2, \sigma\leq2^{-6}n^{-2}p^{-\frac{1}{2}}$,
there exist constants $C^*$ and $\delta>0$  depending only on $M_0$ and $n$,
 such that for any $t\in[0,T)$
we have the  estimate
\[(\int_{M_t} f_{\sigma}^pd\mu)^{\frac{1}{p}}\leq C^*e^{-\delta t}.\]
\end{lm}
\textbf{Proof of Lemma 3.3}
From (\ref{3.12}), it's easy to show
\begin{align*}
&\frac{\partial}{\partial t}\int f_{\sigma}^pd\mu+p(p-1)\int f_{\sigma}^{p-2}\|\nabla f_{\sigma}\|^2d\mu
+\frac{p}{8}\int W^{\sigma-1}f_{\sigma}^{p-1}\|\nabla H\|^2d\mu\\
\leq&4ap \int\|H\|W^{-1}\|\nabla H\|\|\nabla f_{\sigma}\|f_{\sigma}^{p-1}d\mu
+2\sigma p\int(\|A\|^2+n)f_{\sigma}^pd\mu-\frac{p}{2} \int f_{\sigma}^pd\mu\\
&+Cpe^{-\lambda t}\int W^{\sigma}f_{\sigma}^{p-1}d\mu
+\int\|g^{ij}(\bar{R}_{ij}-\frac{\bar{r}\bar{g}_{ij}}{n+1})\|f_{\sigma}^pd\mu
-\int H^2f_{\sigma}^pd\mu\\
\leq&4ap \int\|H\|W^{-1}\|\nabla H\|\|\nabla f_{\sigma}\|f_{\sigma}^{p-1}f_{\sigma}^pd\mu
+(8n\sigma p+1)\int Wf_{\sigma}^pd\mu-\frac{p}{2} \int f_{\sigma}^pd\mu\\
&+Cpe^{-\lambda t}\int W^{\sigma}f_{\sigma}^{p-1}d\mu-\int Wf_{\sigma}^pd\mu,
\end{align*}
where we have used (\ref{2.19}) and the fact that
\[\|A\|^2+n\leq \alpha_n H^2+1+n\leq 4nW,
\ \ f_{\sigma}\leq W^{\sigma}\]
Set $\theta=\frac{1}{16\sqrt{p}}$ in Lemma 3.4, then by our choice of $p$ and $\sigma$,
\begin{align*}
\frac{\partial}{\partial t}\int f_{\sigma}^pd\mu\leq
-\frac{p}{2} \int f_{\sigma}^pd\mu
+Cpe^{-\lambda t}\int W^{\sigma}f_{\sigma}^{p-1}d\mu-\int Wf_{\sigma}^pd\mu\tag{3.18}\label{3.18}
\end{align*}
Since
\begin{align*}
\frac{d}{dt}\int_{M_t}d\mu=&\int_{M_t}(-H^2-g^{ij}\bar{R}_{ij}+\frac{n\bar{r}}{n+1})d\mu\\
\leq&\int_{M_t}\|g^{ij}(\bar{R}_{ij}-\frac{\bar{r}\bar{g}_{ij}}{n+1})\|d\mu
\leq ne^{-\lambda t}\int_{M_t}d\mu
\end{align*}
then
\[\int_{M_t}d\mu\leq e^{\frac{n}{\lambda}}\int_{M_0}d\mu.\]
Let
\[\Lambda=1+e^{\frac{n}{\lambda}}\int_{M_0}d\mu,\] and
\[I=\{t\in[0,T)\big\|\int_{M_t}f_{\sigma}^p> \Lambda e^{-\frac{p\lambda t}{2}}\}.\]
If $I=\emptyset$, then the Lemma follows automatically. Otherwise, let $t_0=\inf I$,
then at $t=t_0$, we have
\[\int_{M_{t_0}}f_{\sigma}^p=\Lambda e^{-\frac{p\lambda t_0}{2}}.\]
For any $t_1\in(t_0,T)$, we only need to consider the case
\[\int_{M_t} f_{\sigma}^p\geq \Lambda e^{-\frac{p\lambda t}{2}}, \ \ \  \forall t \in[t_0,t_1] \]
Let $s\in(1,\frac{3}{2})$  satisfy $\sigma+\frac{1}{s}=1$. Then by H\"{o}lder inequality and
Young's inequality,
\begin{align*}
\int_{M_t} W^{\sigma}f_{\sigma}^{p-1}&=\int_{M_t} W^{\sigma}f_{\sigma}^{p\sigma}f_{\sigma}^{p-p\sigma-1}\\
&\leq(\int_{M_t} Wf_{\sigma}^p)^{\sigma}
(\int_{M_t} f_{\sigma}^{(p-p\sigma-1)s})^{\frac{1}{s}}\\
&\leq \tau\int_{M_t} Wf_{\sigma}^p+(\frac{1}{\tau})^{s\sigma}\int_{M_t} f_{\sigma}^{p-s}\\
&\leq\tau\int_{M_t} Wf_{\sigma}^p+\Lambda(\frac{1}{\tau})^{s\sigma}
(\int_{M_t} f_{\sigma}^p)^{1-\frac{s}{p}}\tag{3.19}\label{3.19}
\end{align*}
By choosing $\tau=\frac{1}{Cp}$ and substituting (\ref{3.19}) into (\ref{3.18}),
\[
\frac{\partial}{\partial t}\int_{M_t} f_{\sigma}^p\leq-\frac{p}{2}\int_{M_t} f_{\sigma}^pd\mu
+2\Lambda Cpe^{-\lambda t}(\int_{M_t} f_{\sigma}^pd\mu)^{1-\frac{s}{p}}\]
Since $\int_{M_t}f_{\sigma}^p d\mu\geq \Lambda e^{-\frac{p\lambda t}{2}}$,  we have
\begin{align*}
\frac{\partial}{\partial t}\int_{M_t} f_{\sigma}^pd\mu\leq
-\frac{p}{2} \int_{M_t} f_{\sigma}^pd\mu
+2\Lambda Cpe^{-\frac{\lambda t}{4}}\int_{M_t} f_{\sigma}^p
\end{align*}
Integrating the inequality above from $t_0$ to $t_1$ yields
\[\int_{M_{t_1}}f_{\sigma}^p\leq e^{\frac{8\Lambda C}{\lambda}p}
\int_{M_{t_0}} f_{\sigma}^pe^{-\frac{p}{2} (t_1-t_0)}\leq
(C^*)^pe^{-p\delta t_1}\]
where $C^*=\Lambda e^{\frac{8\Lambda C}{\lambda}}, \delta=\min\{\frac{\lambda}{2}, \frac{1}{2}\}$.
Note the constant $C^*$ is independent of $t_1$, thus we complete the proof of  the Lemma.

As a consequence of Lemma 3.3, we have
\begin{coro}\label{coro 3.4}
For any $m\geq1$,  $p\geq m^22^7n^2$, and $ \sigma\leq2^{-7}n^{-2}p^{-\frac{1}{2}}$, we have
\[(\int_{M_t} W^mf_{\sigma}^p)^{\frac{1}{p}}\leq C^*e^{-\delta t}\]
\end{coro}
\quad To prove Theorem 3.1, it suffices to give an uniformly upper bound for $g_{\sigma}=f_{\sigma}e^{\frac{\delta t}{2}}$.
For any $m,p,\sigma$ satisfying the condition of Corollary $3.1$,
\begin{equation}\tag{3.20}\label{3.20}
(\int_{M_t} W^mg_{\sigma}^p)^{\frac{1}{p}}\leq e^{\frac{\delta t}{2}}(\int_{M_t} W^mf_{\sigma}^p)^{\frac{1}{p}}
\leq C^*e^{-\frac{\delta t}{2}}.
\end{equation}
Let
$g_{{\sigma},k}=\max(g_{\sigma}-k,0), \phi=g_{{\sigma},k}^{\frac{p}{2}},
A(k,t)=\{x\in M_t\|g_{\sigma}>k\}$,
and
\[\|\|A(k,t)\|\|_{T_1}=\int_{0}^{T_1}\|A(k,t)\|dt=\int_{0}^{T_1}\int _{A(k,t)}d\mu dt\]
By H\"{o}lder inequality,
\[\|A(k,t)\|\leq\frac{1}{k}\int_{M_t}g_{\sigma}d\mu\leq
\frac{\Lambda}{k}e^{\frac{\delta t}{2}}(\int_{M_t}f_{\sigma}^p)^{\frac{1}{p}}d\mu
\leq\frac{\Lambda C^*}{k}\]
From (\ref{3.20}), we derive
\begin{align*}
\int_{A(k,t)}\|H\|^nd\mu\leq (\frac{1}{a})^{\frac{n}{2}}k^{-p}\int_{M_t}W^{\frac{n}{2}}g_{\sigma}^pd\mu
\leq(2n)^n(\frac{C^*}{k})^p\tag{3.21}\label{3.21}
\end{align*}
Given $p\geq 2$, we can choose $k_1\geq2nCC^*$ large enough such that for any $k\geq k_1$
the following Sobolev inequality \cite{Ho} holds
\begin{equation}\notag
(\int_{A(k,t)}\phi^{\frac{n}{n-1}}d\mu)^{\frac{n-1}{n}}\leq c_n(\int_{A(k,t)}\|\nabla \phi\|d\mu+\int_{A(k,t)}\|H\|\phi d\mu)
\end{equation}
where $c_n$ is a constant only depending on $n$. By H\"{o}lder inequality,
\begin{align*}
(\int_{A(k,t)} \phi^{2q}d\mu)^{\frac{1}{q}}\leq& c_n\int_{A(k,t)} \|\nabla \phi\|^2d\mu
+c_n(\int_{A(k,t)} \|H\|^nd\mu)^{\frac{2}{n}}
(\int_{A(k,t)}\phi^{2q}d\mu)^{\frac{1}{q}} \tag{3.22}\label{3.22}
\end{align*}
where
\[
q=\left\{
\begin{array}{lll}
\frac{n}{n-2},\  n>2,\\
<\infty, n=2.
\end{array}
\right.
\]
Since $k_1$ is large enough,
\[c_n(\int_{A(k,t)} \|H\|^nd\mu)^{\frac{2}{n}}\leq\frac{1}{2}\]
By (\ref{3.12}), we have
\begin{equation}\tag{3.23}\label{3.23}
\frac{\partial}{\partial t}\int_{A(k,t)}\phi^2+\int_{A(k,t)}\|\nabla \phi\|^2
\leq Cp\int_{A(k,t)}Wg_{\sigma}^pd\mu
\end{equation}
Substituting (\ref{3.22}) into (\ref{3.23}) gives
\[\frac{\partial}{\partial t}\int_{A(k,t)}\phi^2d\mu+ \frac{1}{C}(\int_{A(k,t)}\phi^{2q}d\mu)^{\frac{1}{q}}
\leq Cp\int_{A(k,t)}Wg_{\sigma}^pd\mu\]
Then for any $T_1<T$,
\begin{equation}\tag{3.24}\label{3.24}
\sup_{[0,T_1]}\int_{M_t}\phi^2d\mu+\frac{1}{C}\int_0^{T_1}(\int_{M_t}\phi^{2q}d\mu)^{\frac{1}{q}}dt\leq Cp\int_0^{T_1}
\int_{A(k,t)}Wg_{\sigma}^pd\mu dt
\end{equation}
Using interpolation inequalities for $L^p$-space, we have
\begin{equation}\notag
(\int_{M_t}\phi^{2q_0}d\mu)^\frac{1}{q_0}\leq
(\int_{M_t}\phi^{2q}d\mu)^{\frac{\eta}{q}}(\int_{M_t}\phi^2d\mu)^{1-\eta}
\end{equation}
where $1<q_0<q$ and  $\eta=\frac{1}{q_0}=\frac{1}{2-\frac{1}{q}}$.

Thus
\begin{align*}
&(\int_0^{T_1}\int_{A(k,t)}\phi^{2q_0}d\mu dt)^{\frac{1}{q_0}}\\
\leq&[\int_0^{T_1}(\int_{A(k,t)}\phi^{2q}d\mu)^{\frac{\eta q_0}{q}}(\int_{A(k,t)}\phi^2d\mu)^{(1-\eta)q_0}dt]^{\frac{1}{q_0}}\\
\leq&(\sup_{[0,T_1]}\int_{M_t}\phi^2d\mu)^{1-\frac{1}{q_0}}(\int_0^{T_1}(\int_{M_t}\phi^{2q}d\mu)^{\frac{1}{q}}dt)^{\frac{1}{q_0}}\\
\leq&Cp\int_0^{T_1}\int_{A(k,t)}Wg_{\sigma}^pd\mu dt\\
\leq&Cp\|A(k,t)\|^{1-\frac{1}{\theta}}(\int_0^{T_1}\int_{M_t}W^{\theta}g_{\sigma}^{p\theta}d\mu dt)^{\frac{1}{\theta}}
\end{align*}
where $\theta>1$ is a positive constant to be chosen, we have used (\ref{3.24}) in the third inequality.
Applying H\"{o}lder inequality again, we have
\begin{align}
\int_0^{T_1}\int_{M_t}\phi^2d\mu dt\leq
&\|A(k,t)\|^{1-\frac{1}{q_0}}
(\int_0^{T_1}\int_{A(k,t)}\phi^{2q_0}d\mu dt)^{\frac{1}{q_0}}\nonumber\\
\leq&Cp\|A(k,t)\|^{2-\frac{1}{q_0}-\frac{1}{\theta}}(\int_0^{T_1}\int_{A(k,t)}W^{\theta}
g_{\sigma}^{p\theta}d\mu dt)^{\frac{1}{\theta}}\tag{3.25}.\label{3.25}
\end{align}
Now we choose $\theta$ large such that $\gamma=2-\frac{1}{q_0}-\frac{1}{\theta}>1$.
 Notice that $\theta$ is independent of the choice of $p$.  Choosing fixed
 $p_1\geq \theta 2^{8}n^2$ and $\sigma_1\leq2^{8}n^{-2}p_1^{-\frac{1}{2}}$,
by (\ref{3.20}), we have
\begin{align*}
&(\int_0^{T_1}\int_{A(k,t)} W^{\theta}g_{\sigma_1}^{p_1\theta}d\mu dt)^{\frac{1}{\theta}}\leq
(\int_0^{T_1}(Ce^{-\frac{\delta t}{2}})^{p_1\theta}dt)^{\frac{1}{\theta}}\leq C_{p_1}\tag{3.26}\label{3.26}
\end{align*}
Together (\ref{3.25}) with (\ref{3.26}), we obtain
\[\|h-k\|^{p_1}\|A(h,t)\|_{T_1}\leq\int_0^{T_1}\int_{M_t}\phi^2d\mu dt\leq C_{p_1}\|A(k)\|_{T_1}^{\gamma},\ \ \ \   h>k>k_1.\]
Thus by the De Giorgi's iteration Lemma, we conclude
\[   \|A(k,t)\|=0, \ \ \ \forall k\geq k_1+d,   \]
where $d$ is a constant depending on $M_0, n$ and  $\lambda $.
Hence
\[g_{\sigma}\leq k_1+d.\]
Notice both $k_1$ and $d$
are independent of $T_1$, so we finish the proof.
\end{proof}

\section{The gradient estimate}

In this section, we use Theorem 3.1 to get an estimate for the gradient of mean curvature.
\begin{thm}\label{thm 4.1}
For any $0<\beta\leq1$, there exists a constant $C_{\beta}$ depending only on $ \bar{g}_0, M_0, n$ and $\beta$,
such that for any point $(x,t)\in M_t\times[0,T)$ we have
\[\|\nabla H\|^2\leq(\beta\|H\|^4+C_{\beta})e^{-\frac{\delta_1}{2}t}\]
\end{thm}
\begin{proof}
By Lemma \ref{lm 2.3},
\begin{align}
\frac{\partial}{\partial t}\|\nabla H\|^2=&
-(\frac{\partial}{\partial t}g_{ij})\nabla_iH\nabla_jH+2\nabla_iH\nabla_i(\frac{\partial H}{\partial t})\nonumber\\
=&(2Hh_{ij}+2\bar{R}_{ij}-\frac{\bar{r}}{n+1}g_{ij})\nabla_iH\nabla_jH+
2\nabla_iH\nabla_i(\Delta H+\|A\|^2H+u)\nonumber\\
\leq& \Delta\|\nabla H\|^2-2\|\nabla^2H\|^2+2\|A\|^2\|\nabla H\|^2+2H\nabla_iH\nabla_i\|A\|^2+2h_{ij}h_{jp}\nabla_iH\nabla_pH\nonumber\\
&-2\bar{R}_{ijpj}\nabla_iH\nabla_pH+C\|\nabla A\|^2+2\|\nabla H\|\|\nabla u\|\nonumber\\
\leq&\Delta\|\nabla H\|^2+C_2(H^2+1)\|\nabla A\|^2+C_2e^{-\lambda t}\tag{4.1}\label{4.1}
\end{align}
and
\begin{align*}
&\frac{\partial}{\partial t}(H^2(\|A\|^2-\frac{H^2}{n}))\\
=&H^2\frac{\partial}{\partial t}(\|A\|^2-\frac{H^2}{n})
+(\frac{\partial }{\partial t}H^2)(\|A\|^2-\frac{1}{n}H^2)\\
=&H^2\{\Delta(\|A\|^2-\frac{1}{n}H^2)-2(\|\nabla A\|^2-\frac{\|\nabla H\|^2}{n})+2\|A\|^2(\|A\|^2-\frac{H^2}{n})
+(v-\frac{2}{n}uH)\}\\
&+(\Delta H^2-2\|\nabla H\|^2+2\|A\|^2H^2+2uH)(\|A\|^2-\frac{H^2}{n})\\
=&\Delta(H^2(\|A\|^2-\frac{H^2}{n}))
-4H\nabla_iH\nabla_i(\|A\|^2-\frac{H^2}{n})-2H^2(\|\nabla A\|^2-\frac{\|\nabla H\|^2}{n})\\
&-2(\|A\|^2-\frac{H^2}{n})\|\nabla H\|^2
+4\|A\|^2H^2(\|A\|^2-\frac{H^2}{n})
+H^2(v-\frac{2}{n}uH)+2uH(\|A\|^2-\frac{H^2}{n})\tag{4.2}\label{4.2}
\end{align*}
By applying Theorem 3.1 we can give an  estimate of the second term of (\ref{4.2}) ,
\begin{align*}
\|4H\nabla_iH\nabla_i(\|A\|^2-\frac{H^2}{n})\|=&\|8H\nabla_iHh_{kl}^0\nabla_ih_{kl}^0\|\\
\leq&8\|H\|\|\nabla A\|\|\nabla H\|\|h_{kl}^0\|\\
\leq&8n\|H\|\|\nabla A\|^2C_1(H^2+1)^{\frac{1-\sigma}{2}}\\
\leq&\frac{1}{4n}H^2\|\nabla A\|^2+C_2\|\nabla A\|^2\tag{4.3}\label{4.3}
\end{align*}
By Lemma \ref{lm 2.4}, we can choose $\eta>0$ such that
\begin{equation}\tag{4.4}\label{4.4}
\|\nabla A\|^2-\frac{1}{n}\|\nabla H\|^2\geq\frac{1}{4n}\|\nabla A\|^2-C_2e^{-\lambda t}
\end{equation}
Combining (\ref{3.2}) and (\ref{3.3}) gives
\begin{equation}\tag{4.5}\label{4.5}
H^2(v-\frac{2}{n}uH)+2uH(\|A\|^2-\frac{H^2}{n})\leq C_2(\|A\|^4+1)e^{-\lambda t}
\end{equation}
Substituting (\ref{4.3}), (\ref{4.4}) and (\ref{4.5}) into (\ref{4.2}) yields
\begin{align}
\frac{\partial}{\partial t}(H^2(\|A\|^2-\frac{H^2}{n}))\leq&
\Delta(H^2(\|A\|^2-\frac{H^2}{n}))-\frac{1}{4n}H^2\|\nabla A\|^2+4\|A\|^2H^2(\|A\|^2-\frac{H^2}{n})\nonumber\\
&+C_2\|\nabla A\|^2+C_2(\|A\|^4+1)e^{-\lambda t}\tag{4.6}\label{4.6}
\end{align}
Similarly, we can get the following estimate
\begin{align}
&\frac{\partial}{\partial t}(\|A\|^2-\frac{H^2}{n})\nonumber\\
=&\Delta (\|A\|^2-\frac{H^2}{n})
-2(\|\nabla A\|^2-\frac{\|\nabla H\|^2}{n})+2\|A\|^2(\|A\|^2-\frac{H^2}{n})+v-\frac{2}{n}uH\nonumber\\
\leq&\Delta (\|A\|^2-\frac{H^2}{n})-\frac{1}{2n}\|\nabla A\|^2+2\|A\|^2(\|A\|^2-\frac{H^2}{n})
+C_2(\|A\|^2+1)e^{-\lambda t}\tag{4.7}\label{4.7}
\end{align}
Let $\Psi=H^2(\|A\|^2-\frac{H^2}{n})+4nC_2(\|A\|^2-\frac{H^2}{n})$. Then it follows from (\ref{4.6}) and (\ref{4.7})
\begin{align*}
\frac{\partial}{\partial t}\Psi\leq&
\Delta \Psi-(\frac{H^2}{4n}+C_2)\|\nabla A\|^2+4\|A\|^2(\|A\|^2-\frac{H^2}{n})(H^2+2nC_2)\\
&
+5nC_2^2(\|A\|^4+1)e^{-\lambda t}
\end{align*}
A direct computation shows
\begin{align}
\frac{\partial}{\partial t}\|A\|^4=&2\|A\|^2(\Delta \|A\|^2-2\|\nabla A\|^2+2\|A\|^4+v)\nonumber\\
\geq&\Delta \|A\|^4-12\|A\|^2\|\nabla A\|^2+4\|A\|^6-4n\|A\|^2(\|A\|^2-\frac{1}{n}H^2)\nonumber\\
&-
C_2(\|A\|^4+1)e^{-\lambda t}\tag{4.8}\label{4.8}
\end{align}
Now consider the function
\[\Phi=e^{\frac{\delta_1 t}{2}}(\|\nabla H\|^2+C_3\Psi)-\beta\|A\|^4\]
Choose $C_3\geq12nC_2$, then  there exists a constant $C_4$  such that
\[\frac{\partial}{\partial t}\Phi\leq\Delta \Phi+C_4(\|A\|^4+1)(\|A\|^2-\frac{H^2}{n})e^{\frac{\delta_1 t}{2}}
+C_4(\|A\|^4+1)e^{-\lambda t}-4\beta\|A\|^6\]
Applying Theorem \ref{thm3.1} and  Young's inequality, we obtain
\[\frac{\partial}{\partial t}\Phi\leq\Delta\Phi+C_5e^{-\frac{\delta_1 t}{2}}\]
Therefore $\Phi$ is bounded by a constant $C_6$.  Hence
\[\|\nabla H\|^2\leq(\beta\|A\|^4+C_6)e^{-\frac{\delta_1 t}{2}}.\]
We complete the proof by Lemma 2.6.
\end{proof}

\section{Convergence of the hypersurface}
\quad In this section, we use Theorem 3.1 and Theorem 4.1 to finish the proof of Theorem 1.1.

\begin{proof}[Proof of Theorem 1.1]
We consider the following two cases.

\textbf{Case 1}:
 $\max_{M_t}\|H\|\rightarrow \infty$ as $t\rightarrow T$.  By Theorem \ref{thm 4.1},
we always have
\[\|\nabla H\|\leq\beta^2 H^2+C_{\beta}\ \ \ \text{on}\ \  t\in[0,T).\]
Let
\[H_{\max}(t)=\max_{M_t}H, \ \ H_{\min}(t)=\min_{M_t}H
\]
 Suppose
$\max_{M_t}\|H\|=H_{\max}(t)>0$.
For any $\beta>0$, there exists some $\theta$ depending on $\beta$ with
$C_{\beta}\leq\beta^2 H_{\max}^2$ at $t=\theta$, so
$\|\nabla H\|\leq2\beta^2 H_{\max}^2$.
Let $x_0$ be the point where $H$ attains its maximum.
Then for any point $x$ with $d(x,x_0)\leq\frac{1}{\beta H_{\max}}$, we have
\[H(x)\geq H_{\max}-2d(x,x_0)\beta^2 H_{\max}^2\geq(1-2\beta) H_{\max}\]
and the sectional curvature $K_{M_t}(x)$ of $M_t$ satisfies
\[K_{M_t}(x)\geq\frac{ H^2}{8n^2}\geq\frac{(1-2\beta)^2{H_{\max}^2}}{{8n^2}},
\]
By Myers' theorem, any geodesic starting from $x_0$ with length larger than
$\frac{2\sqrt{2}n\pi}{1-2\beta}H_{\max}^{-1}$ must have conjugate points.
By choosing $\beta$ small, we can get
\[H_{\min}\geq(1-2\beta)H_{\max}\ \ \ \text{on}\ \ \  M_{\theta}\]
Thus by a suitable choice of $\theta$ we know the mean curvature of the hypersurface is positive and can be
arbitrarily large. Moreover, at some $t=\theta$ the inequality  below holds everywhere on $M_{\theta}$
\[\|A\|^2\leq\alpha_n H^2+1<\frac{1}{n-1}H^2\]
Hence $M_{\theta}$ is strictly  convex.  By the maximum principle,  the maximal existence time of the
equation (\ref{1.2}) must be finite.  By a similar argument as Huisken in \cite{Hu3},
 we know $M_t$ converge  to a round point.

\textbf{Case 2}:
$\|H\|$ is  uniformly bounded and $T=\infty$. Now
\[\|A\|^2-\frac{H^2}{n}\leq Ce^{-\delta_1 t},\ \ \ \|\nabla H\|^2\leq Ce^{-\frac{\delta_1 t}{2}}\]
Furthermore, we can get

\textbf{Claim :} $H_{\max}(t)>- \tilde{C}e^{ -\frac{\delta_1 t}{4}}$ and
$H_{\min}(t)< \tilde{C}e^{ -\frac{\delta_1 t}{4}}$ for some large constant $\tilde{C}>C$.

Suppose there exists a moment $t_0$ such that $H_{\max}(t)\leq-\tilde{C}e^{ -\frac{\delta_1 t}{4}}$ at $t=t_0$.
Note $\delta_1\leq \frac{\lambda}{2} $. From (3) in Lemma \ref{lm 2.5}, at $t=t_0$, we have
\begin{align*}
\frac{\partial H_{\max}}{\partial t}&\leq\|A\|^2H_{\max}+H_{\max}+Ce^{-\lambda t}\\
&\leq-\tilde{C}e^{-\delta_1 t}+Ce^{-\lambda t}\\
&<0
\end{align*}
It follows
\[H_{\max}(t)\leq -\tilde{C}e^{ -\frac{\delta_1 t_0}{4}}, \ \forall \  t \in[t_0,\infty)
\]
which contradicts with the fact that $H(\cdot,t)\rightarrow 0$ as $t\rightarrow\infty$.
The other inequality $H_{\min}(t)< \tilde{C}e^{ -\frac{\delta_1 t}{4}}$ can be derived by the same way.

On the other hand, by Lemma \ref{lm 3.2}, the Ricci curvature of $M_t$ is no less than $\frac{1}{8n}$,
thus the diameter of $M_t$ is smaller than $2\sqrt{2}n\pi$. Since
$\|\nabla H\|^2\leq Ce^{-\frac{\delta_1 t}{2}}$, we have
\[\|H_{\max}(t)-H_{\min}(t)\|\leq Ce^{-\frac{\delta_1 t}{4}} \]
Then it follows
$\|H\|^2\leq 4\tilde{C}^{2}e^{-\frac{\delta_1}{2}t}$
and $\|A\|\leq 5\tilde{C}^2e^{-\frac{\delta_1 t}{4}}$.
One can show the exponentially decreasing for $\|\nabla^mA\|$ by the similar argument as \cite{Hu3}.
Since $(N^{n+1},\bar{g}(t))$  converge to  $S^{n+1}$ in $C^{\infty}$-topology,
so we get the $C^{\infty}$-convergence to the totally geodesic sphere for $M_t$.
\end{proof}

\textbf{Addresses:}

\vspace{2mm} Weimin Sheng: Department of Mathematics, Zhejiang University,
Hangzhou 310027, China.

\vspace{2mm} Haobin Yu: Department of mathematics, Zhejiang university,
Hangzhou 310027, China.

\vspace{3mm}

\textbf{Email:} weimins@zju.edu.cn; robin1055@126.com.

\end{document}